\documentclass[12pt]{article}
\usepackage[plainpages=false]{hyperref}
\usepackage{amsfonts,latexsym,rawfonts,amsmath,amssymb,amsthm, mathrsfs, lscape, mathtools}
\usepackage{verbatim}
\usepackage[all]{xy}
\usepackage{authblk}
\usepackage{color}
\usepackage{amsthm}
\usepackage[margin=2cm]{geometry}
\usepackage[backend=biber,style=trad-abbrv,url=false,doi=false,isbn=false,sorting=nyt,maxbibnames=3]{biblatex}
\AtEveryBibitem{\clearfield{month}}
\AtEveryBibitem{%
  \clearlist{language}%
}

\usepackage{array, tabularx}

\usepackage{setspace}
\newtheorem{thm}{Theorem}[section]
\newtheorem{cor}[thm]{Corollary}
\newtheorem{lem}[thm]{Lemma}

\newtheorem{prop}[thm]{Proposition}
\newtheorem{prob}[thm]{Problem}

\newtheorem{rmk}[thm]{Remark}

\numberwithin{equation}{section}\def \N {\mathbb N}

\def \Q {\mathbb Q}

\def \C {\mathbb C}
\def \Z {\mathbb Z}
\def \P {\mathbb P}

\DeclareMathOperator{\qDef}{qDef}

\title{Some observations on the dimension of Fano K-moduli}

\author{Jesus Martinez-Garcia and Cristiano Spotti }

\date{}

\begin{document}

\maketitle

\begin{abstract}In this short note we show the unboundedness of the dimension of the K-moduli space of $n$-dimensional Fano varieties, and that the \emph{dimension of the stack} can also be unbounded while, simultaneously, the dimension of the corresponding coarse space remains bounded.
\end{abstract}

\section{Main statement}

Moduli spaces of K-stable Fano varieties  have been intensively investigated in the last decade, both from a general theory point of view as well as via the study of explicit examples. In this note, we observe the following:

\begin{thm}\label{MT} For each $n>1$ the dimension (as a variety) of the K-moduli spaces of $n$-dimensional Fano varieties is unbounded. Moreover, the \emph{dimension of the K-moduli stack} can be arbitrarily big, while the dimension of its coarse variety remains bounded. 
\end{thm}

Here, for \emph{dimension of the K-moduli stack} at a given point $X$ we mean the difference between the dimension of the versal space of (K-semistable) $\Q$-Gorenstein deformations of minus the dimension of the reductive automorphism group at a K-polystable point $X$ \cite[Section 0AFL]{stacks-project}.

It is well-known that smooth Fano manifolds, and more generally $\varepsilon$-log terminal Fano varieties (where $\varepsilon>0$ is fixed), form a bounded family in a fixed dimension \cite{birkar2020singularities}. Thus to construct such examples we need to consider non-smoothable varieties whose Kawamata log terminal (klt) singularities get worse and worse.

Our main theorem is a quick consequence of these two dimensional easy examples.

\begin{prop}\label{MP}
Consider the following two families of K-polystable normal surfaces:
\begin{enumerate}
    \item $X_l:= (\P^1\times \P^1)/\Z_l$, for $l\geq2$, where the action is generated by $$\zeta.([z_0:z_1],[w_0:w_1])=([\zeta z_0:z_1], [\zeta^{-1} w_0:w_1]),$$ where $\zeta$ is a primitive $l$-root of unity.
    \item $Y_l:= \P^2/\Z_l$ for $l\geq3$, $l$ odd, where the action is generated by $$\zeta.[z_0:z_1:z_2]=[\zeta z_0: \zeta^{-1} z_1: z_2],$$ where $\zeta$ is a primitive $l$-root of unity. 
    
\end{enumerate} Then \begin{enumerate}
    \item the dimension as a variety of the K-moduli space $M^K$ at $[X_l]$ is equal to $2l-3$ if $l\neq 2,4$, and equal to $2$ (resp. $6$) for $l=2$ (resp. $l=4$). 
    \item the dimension of the K-moduli stack $\mathcal{M}^K$ at $Y_l$ is equal to $l-3$ for $l\neq 3, 9$ and equal to $4$ (resp. $8$) for $l=3$ (resp. $9$). However, $[Y_l]$ is an \emph{isolated} $K$-polystable point for $l\neq 3, 9$.
\end{enumerate}
\end{prop}

The surfaces $X_l$ with $l=2,4$ and $Y_l$ with $l=3,9$ are actually $\Q$-Gorenstein smoothable and they appear in the boundary of $K$-moduli of smooth del Pezzo surfaces of degree $4$, $2$,  $3$ and $1$ respectively \cite{OSS}. 

Note that if we would have considered $Y_l$ for $l$ even, we would have  a $\Z_2$ subgroup fixing the line $z_2=0$ (thus the more natural way to think about the quotient is as a pair $(X,D)$, considering the line at infinity with weight $\frac{2}{l}$). This pair will be only \emph{log}-K-polystable, not $X$.

\begin{proof}[Proof of Theorem \ref{MT}]
It simply follows by taking $\tilde{{X}_l}=X_l \times \P^{n-2}$. Of course, being the product of two $K$-polystable varieties, such $n$-dimensional varieties are still $K$-polystable \cite{Zhuang-product-k-stability} and hence the dimension of the $K$-moduli spaces at $[\tilde{X}_l]$ tends to infinity with $l$. Note, moreover, that $Y_l$ for $l\neq3,9$ actually give examples where the K-moduli reduces to a point while there are many non isomorphic strictly K-semistable Fano varieties around $Y_l$ (we are unaware if a similar phenomenon can occur for smooth Fano manifolds too).
\end{proof}

These examples also suggest the following problem:

\begin{prob}
Study in detail the local theory of K-moduli of toric del Pezzo surfaces.
\end{prob}

We expect that such investigations are interesting and important when studying moduli spaces of non-necessarily $\Q$-Gorenstein smoothable del Pezzo surfaces.

The proof of the above Proposition is based on the local study of K-stability for $\mathbb Q$-Goreinstein deformations of  the surfaces, which is possible even in this non-smoothable setting thanks to the the recent works \cite{ blum-liu-xu-ksemistability-openness,blum-HalpernLeistner-liu-xu-properness}. These type of computations have been performed for the $\Q$-Gorenstein smoothable cases of the above examples in \cite{OSS}. A similar strategy to show interesting behaviour of K-moduli spaces near toric varieties has also been considered in \cite{Kaloghiros-Petracci} to show that the moduli can be reducible and non-reduced.

\subsection*{Acknowledgments} Some of the ideas of this work were fostered in the series of conferences titled Birational Geometry, K\"ahler-Einstein Metrics and Degenerations, taking place in Moscow, Shanghai and Pohang in 2019, the first two of which were attended by the first author, who would like to thank the organisers for inviting him. The second author is supported by Villum Young Investigator $0019098$.

After writing up a first draft of this problem in late December 2020, we found out that the first example in Proposition \ref{MP} was considered a few weeks before in \cite{Liu-Petracci}, when studying the K-stability of hypersurfaces in $\P(1,1,a,a)$. We would like to thank A. Petracci for having a look at an early draft of our manuscript and giving us very useful comments which improved our manuscript.

\section{Proof of Proposition \ref{MP}}

Proposition \ref{MP} is a consequence of the next few lemmas.

\begin{lem}\label{lem:sing-classification} For the surface $X_l$ above we have that $\mbox{Sing}(X_l)=\{2A_{l-1},2 \frac{1}{l}(1,1)\}$, and the connected component to  the identity is $\mbox{Aut}_0(X_l)=(\C^{*})^2$. Similarly for $Y_l$ we have that $\mbox{Sing}(Y_l)=\{A_{l-1},2\frac{1}{l}(1,2)\}$ and $\mbox{Aut}_0(Y_l)=(\C^{*})^2$. 
\end{lem}

\begin{proof}
Let's consider the $Y_l$ case ($X_l$ is completely analogous and we omit it). The singularities of $Y_l$ correspond to points on $\P^2$ where $\Z_l$ acts with non-trivial stabilizer. Near $[0:0:1]$ the action has weight $(1,-1)$ resulting in a $A_{l-1}$ canonical singularity. Similarly near the points $[1:0:0]$ and $[0:1:0]$ the action as weight $(1,2)$ resulting in $\frac{1}{l}(1,2)$ quotient singularities The statement about the automorphism follows by noting that $\mathrm{Aut}_0(Y_l)\cong \mathrm{Aut}_0(\mathbb P^2; S)\cong (\mathbb C^*)^2$, where $\mathrm{Aut}_0(\mathbb P^2; S)$ is the fixed component of the automorphism group fixing the subset $S$. For a similar computation see  \cite[Lemma 3.1]{JMG-Calabi-dream}. 
\end{proof}

\begin{rmk}
\label{remark:unboundedness-reason}
Note that the non-Du Val singularities of the set of varieties $\{X_l\}_{l=2}^\infty$ and $\{Y_l\}_{l=3}^\infty$ are indeed not $\epsilon$-log terminal for every $\epsilon>0$. Indeed, each of the two singular points $\frac{1}{l}(1,1)$ in $X_l$ is  locally analytically isomorphic to  the affine cone over the rational normal curve $C_l\subset\mathbb P^l$ and its resolution has exceptional locus $E\cong \mathbb P^1$ with $E^2=-l$. It follows that their log discrepancies equal to $\frac{2}{l}-1 \rightarrow -1$, and moreover $-K_{X_l}$ is $\Q$-Cartier (with Cartier index going to infinity) with $(-K_{X_l})^2=\frac{8}{l}\rightarrow 0$, and similarly for $Y_l$.
\end{rmk}

\begin{lem}
\label{lem:qDef-computation}
$X_l$ and $Y_l$ are K-polystable Fano variety whose space of $\Q$-Gorenstein deformations  are given by

\begin{enumerate}
    \item  $\qDef(X_l)\cong \qDef(A_{l-1})\oplus \qDef(A_{l-1}) \cong  \C^{2(l-1)}$ for $l\neq 2,4$.
    \item $\qDef(Y_l) \cong \qDef(A_{l-1}) \cong  \C^{(l-1)}$ for $l\neq 3,9$.
\end{enumerate}

\end{lem}

\begin{proof}[Proof of Lemma \ref{lem:qDef-computation}]
$Y_l$ (and $X_l$) is a K-polystable del Pezzo surface as $\Z_l$ acts by isometries with respect to the Fubini-Study metric in $\P^2$ (with respect to the product in $\P^1\times \P^1$ of the product of the Fubini-Study metrics in $\P^1$, respectively). Hence, both $X_l$ and $Y_l$ inherit an (orbifold) K\"ahler-Einstein metric and consequently they are K-polystable by \cite{Berman-KE-Kps-Klt}.

By \cite[Lemma 6]{Corti-Petracci-et-al-surfaces-ms}, it follows that there are no local-to-global obstructions to $\Q$-Gorenstein deformations on del Pezzo surfaces. Since $X_l$ is toric it does not admit \emph{equisingular deformations} (i.e. non-trivial deformations to a non-isomorphic projective variety with the same singularities), e.g., \cite[Lemma 4.4]{Petracci}. Hence all $\Q$-Gorenstein deformations must come from local $\Q$-Gorenstein deformations of the singularities. Thus
\begin{equation}
\label{eq:def-breakdown}
    \qDef(Y_l)=\bigoplus_{p\in \mathrm{Sing}(Y_l)}\qDef(p)
\end{equation}
and similar for $X_p$. 

Note that any deformation of $A_{l-1}$ is $\mathbb Q$-Gorenstein and given by the versal family $xy=z^{l}+a_{l-2}z^{l-2}+ \dots +a_0$. Hence the vector $(a_0, a_1, \dots, a_{l-2})$ defines a point in $\qDef(A_{l-1})$ and $\qDef(A_{l-1})\cong \C^{l-1}$. The proof follows from Lemma \ref{lem:sing-classification}, once we show that $\qDef(p)=\{0\}$ for $p$ non-Du Val. We will do this for $Y_l$, since the case of $X_l$ is very similar.

We claim the two $\frac{1}{l}(1,2)$ singularities of $Y_l$ are $\Q$-Gorenstein rigid (i.e. they do not admit $\Q$-Gorenstein deformations) if $l\neq 3,9$, and $\Q$-Gorenstein smoothable otherwise. Let $w=\mathrm{hcf}(l,3)$, $r> 0$ such that $l=wr$, $m\geqslant 0$ and $0\geqslant w_0<r$ such that $w=mr+w_0$. It is well known (see e.g. \cite{Corti-Petracci-et-al-surfaces-ms}) that a quotient singularity $\frac{1}{l}(1,2)$ is $\Q$-Gorenstein rigid if and only if $m=0$, or equivalently if $w=w_0$. Moreover, $\frac{1}{l}(1,2)$ is $\Q$-Gorenstein smoothable (often known as a \emph{T-singularity}) if and only if $w_0=0$ and a \emph{primitive T-singularity} if in addition $m=1$.

The number $w=\mathrm{hcf}(l,3)$ can only be $1$ or $3$. If $w=1$, then $l=wr=r$ and $1=w=mr+w_0$ implies that $m=0$ so $\frac{1}{l}(1,2)$ is $\Q$-Gorenstein rigid. If $w=3$ then $l=3k$ for some $k\in \N$ but in fact, that means that $l=3k=wr=3r$, so $l=3r$. If $r=1$ then $m=1$ and $w_0=0$ so $\frac{1}{3}(1,2)$ is a primitive T-singularity. The case $r=2$ is excluded, otherwise $l$ would be even. If $r=3$, then $m=1$ and $w_0=0$ and $\frac{1}{9}(1,2)$ is $\Q$-Gorenstein smoothable. If $r\geqslant 4$, ($r$ odd) then $m=0$ and $\frac{1}{3r}(1,2)$ is $\Q$-Gorenstein rigid. Hence, whenever $l\neq 3, 9$ we have $\qDef(\frac{1}{l}(1,2))=\{0\}$.

For $X_l$ similar computations show that the singularities $\frac{1}{l}(1,1)$ are $\Q$-Gorenstein rigid for $l\neq2,4$.
\end{proof}

\begin{rmk}
For $l=2$, $X_l$ has four $A_1$ singularities giving a four dimensional versal space of deformation. For $l=4$, the deformation space has (beside the deformations coming from the two $A_3$ singularities) $\Q$-Gorenstein deformations coming from the one dimensional family of $\Q$-Gorenstein smoothings of the $\frac{1}{4}(1,1)$ singularities. For $l=3$, $Y_l$ is just the unique cubic surface with $3A_2$-singularities, given by $xyz=t^3$ (and the only strictly K-polystable surface in the K-moduli of del Pezzo surfaces of degree $3$). The case $l=9$ was studied in \cite[Example 3.10]{OSS} and it appears in the boundary of the K-moduli compactification of smooth del Pezzo surfaces of degree $1$.
\end{rmk}


\begin{lem} \label{LA} The natural action of $G=Aut_0(X_l)\cong (\C^{\ast})^2$ ($G'=Aut_0(Y_l)\cong (\C^{\ast})^2$) on $\qDef(X_l)$ (respectively $\qDef(Y_l) $) for  $l\neq 2,4$  (resp. $l\neq 3,9$) is not effective. Moreover:
\begin{enumerate}
    \item The action on $\qDef(X_l)\cong \C^{2(l-1)}$ of $G/\cap_x(G_x)\cong \C^*$ with $t=\lambda_1\lambda_2\in G/\cap_x(G_x)$,  is given by
$$(a_0, a_1,\dots, a_{l-2}, a_{0}^{'}, \dots, a_{l-2}^{'}) \mapsto (t^{l}a_0,t^{l-1}a_1, \dots, t^2 a_{l-2}, t^{-l} a_0^{'}, \dots, t^{-2} a_{l-2}^{'});$$
    \item The action on $\qDef(Y_l)\cong \C^{l-1}$ of $G'/\cap_x(G'_x)\cong \C^*$ with $t=\lambda_1\lambda_2\in G'/\cap_x(G'_x)$,  is given by $$(a_0, a_1,\dots, a_{l-2}) \mapsto (t^{l}a_0,t^{l-1}a_1, \dots, t^2 a_{l-2}).$$
\end{enumerate}
\end{lem}


\begin{proof}
 Let us start with $Y_l$. In local coordinates near the $A_{l-1}$-point $[0:0:1]$ we can take coordinates on  $Aut_0(Y_l)\cong (\C^{\ast})^2$-action such that the action is just given by $$(u,v) \mapsto(\lambda_1^{-1}u , \lambda_2^{-1}v).$$  Taking invariants for the $\Z_l$-action $x=u^l$, $y=v^l$ and $z=uv$, we get the induced action on the $A_{l-1}$-quotient singularity $xy=z^l$ given by $(\lambda_{1}^{-l} x, \lambda_{2}^{-l} y, (\lambda_{1}\lambda_{2})^{-1}z)$. Considering then the natural action induced on the versal deformation family of the singularity $xy=z^{l}+a_{l-2}z^{l-2}+ \dots +a_0$, we get that $$(a_0, a_1, \dots, a_{l-2}) \mapsto ((\lambda_1\lambda_2)^{l}a_0,(\lambda_1\lambda_2)^{l-1}a_1, \dots, (\lambda_1\lambda_2)^2 a_{l-2}).$$ 
In particular note that the action is non effective since the action of the subtorus $(s,s^{-1})\subseteq (\C^{\ast})^2 $ is  clearly trivial. Finally, putting $t=\lambda_1\lambda_2$  we obtain our statement for $Y_l$. 

The statement for $X_l$ is completely analogous, but (crucially) noticing that if we take coordinates on $Aut_0(X_l)$ to be such that near the point $([0:1],[0:1])$ the action is again by  given by $(u,v) \mapsto(\lambda_1^{-1}u , \lambda_2^{-1}v)$, then near the point $([1:0],[1:0])$ one get an action with \emph{opposite} weights. From there the statements follows immediately. \end{proof}

Descriptions of the local actions for the smoothable cases of $X_l$ and $Y_l$ can be found in \cite{OSS}. Also note that since the above action is not effective (with a $\C^\ast$ as stabilizer) all the small deformations will have a residual $\C^\ast$-action on them.

\begin{lem}\label{lem:GIT-picture} When $l\neq2,4$ the K-moduli space near $[X_l]$ is (up to a finite group action) described by the affine GIT quotient $\C^{2(l-1)}// \C^\ast$, where the $\C^\ast$-action is given as in Lemma \ref{LA}. Similarly for $Y_l$ when $l\geqslant 4$ $l\neq 9$, the K-moduli space near $[Y_l]$ is (up to a finite group action) described by the affine GIT quotient $\C^{l-1}// \C^\ast$,
\end{lem}

\begin{proof} Any $\Q$-Gorenstein deformation of $X_l$ and $Y_l$ is still a Fano variety since the canonical $K_{\mathcal{X}}$ of the total space of a deformation $\mathcal{X}$ is $\Q$-Cartier and ampleness is an open condition. Moreover, the deformation is singular, since it is flat and $K_X^2\not\in \mathbb Z$.  Then the characterization of those varieties in the deformation which are  $K$-polystable follows by the local GIT description of non-necessarily smoothable Fano varieties in \cite[Proof of Theorem 4.5]{blum-liu-xu-ksemistability-openness}, cf.\cite[Remark 2.11]{Alper-Blum-HL-Xu}, where it is shown that $K$-semistability is an open condition and that $K$-polystability can be checked locally by considering the action of the automorphisms.
\end{proof}

We are now ready to conclude the proof of our Proposition \ref{MP}:
\begin{proof}[Proof of Proposition \ref{MP}]. For $Y_l$ it is clear that that all points near zero in $\qDef(Y_l)$ are K-semistable by openness. However, note that all such points are destabilized to zero since
$$\lim_{t\rightarrow 0}(t^la_0, \dots, t^2 a_{l-2})=0.$$
Hence only $0$ is GIT polystable, and $Y_l$ an isolated K-polystable variety. However, by \cite[Lemma 98.12.1]{stacks-project}, the dimension of the stack at the point $Y_l$ is equal to $$\mbox{dim}_{Y_l}(\mathcal{M}^k)= \mbox{dim}\,\qDef(Y_l)-\mbox{dim}\, \mathrm{Aut}(Y_l)= (l-1)-2=l-3.$$

For $X_l$ it is now sufficient to compute the dimension (as a variety) of the GIT quotient $\C^{2(l-1)}// \C^\ast$ above. But it is clear that the generic orbit is closed (with no further stabilizer). Indeed, if coordinates $a_j$ and $a_j'$ in Lemma \ref{LA} are all non-zero, then the orbits are given by the closed set $a_ja_{j}^{'}=c_j\neq 0$, with $j=0, \dots, {l-2}$. Hence $\mbox{dim}_{\C}  M^K$ near $[X_l]$ is simply given by $2(l-1)-1=2l-3$ as claimed.
\end{proof}

Observe that if we consider a deformation of $X_l$ which smooths only one of the two $A_{l-1}$ singularities, the resulting variety is strictly K-semistable and never K-polystable, since in order to obtain K-polystable varieties we need to deform the two $A_{l-1}$ singularities \emph{simultaneously} by the same computation as for the $Y_l$ case in the last paragraph of the proof of Proposition \ref{MP}.

Note  also that for $X_l$, $l\neq 2,4$, since the action is not effective, we also have a discrepancy between the dimension of the stack and the dimension of the coarse space (which is then one dimension \emph{bigger} than expected).


\section{Some final comments}

The general small deformation $X_t$ of $X_l$ is then a K-polystable variety which is also K\"ahler-Einstein by \cite{Li-Tian-Wang-YTD-conj-singular}. Moreover the second Betti number gets bigger and bigger as $l$ goes to infinity: indeed, smoothing out an $A_{l-1}$-singularity introduces a chain of $S^2$ of length ${l-1}$, giving distinct homological classes. Hence:

\begin{cor}
 There are K-polystable/K\"ahler-Einstein del Pezzo surfaces with arbitrarily big second Betti number.
\end{cor}

We should also observe that this moduli space corresponds to the moduli of K\"ahler-Einstein orbifolds with positive cosmological constant, hence giving also examples of moduli spaces of positive Einstein orbifolds of unbounded dimension. Thus, from a more differential geometric perspective,  it would be interesting to know if a bound on the second Betti number would instead force the dimension of the moduli spaces of such metrics to stay bounded.

Finally, note that the unboundedness of the dimension can be avoided by bounding below either the volume or the singularities. Indeed, that is what \cite{Jiang-boundedness} proves, where the measure of boundedness used for the singularities is the alpha-invariant. This does not contradict our example, as we had that $K_{X_l}^2\rightarrow 0$ as $l$ grows and the log discrepancies were monotonously decreasing with $l$ towards $-1$. What is remarkable of this example is not that a bound below on the volume or the singularities are required to achieve boundedness of families, there were plenty of examples of this behaviour in \cite{Jiang-boundedness}. What is remarkable is that removing such bounds not only gives an infinite number of families (whose dimension, one may think could, in principle, be uniformly bounded), but it also gives infinite dimension of the moduli.

\printbibliography

		\bigskip	

{\sc University of Essex}

{\tt jesus.martinez-garcia@essex.ac.uk}

{\sc Aarhus University}

{\tt c.spotti@math.au.dk}

\end{document}